\documentclass[12pt]{article}

\usepackage{amsthm,amsmath,amssymb}
\usepackage{graphicx}

\newtheorem{theorem}{Theorem}[section]
\newtheorem{lemma}[theorem]{Lemma}

\newtheorem{conjecture}[theorem]{Conjecture}

\def\nfrac#1#2{{\textstyle\frac{#1}{#2}}}
\def\dfrac#1#2{\lower0.15ex\hbox{\large$\frac{#1}{#2}$}}

\let\originalleft\left
\let\originalright\right
\renewcommand{\left}{\mathopen{}\mathclose\bgroup\originalleft}
\renewcommand{\right}{\aftergroup\egroup\originalright}

\usepackage{enumitem}
\setenumerate{itemsep=10pt}

\usepackage{pgfplots}

\global\long\def\E{\mathbb{E}}
\global\long\def\Pr{\mathbb{P}}
\global\long\def\k{d}

\global\long\def\i{j}

\global\long\def\ii{k}

\global\long\def\iii{\ell}

\global\long\def\falling#1#2{\left(#1\right)_{#2}}
\global\long\def\darrow{\rightsquigarrow}

\global\long\def\one{\boldsymbol{1}}
\global\long\def\G{\mathcal{G}_{n,\k}}
\global\long\def\multiG{\Omega_{n,\k}}
\global\long\def\P{\mathcal{P}_{n,\k}}
\global\long\def\nP{\#P}
\global\long\def\r{\rho}
\global\long\def\Y{Y}
\global\long\def\T{\mathcal{T}_{n}}

\global\long\def\D{\mathcal{D}_{n}}

\global\long\def\d{\delta}

\begin{document}

\title{\bf On the number of spanning trees\\ in random regular graphs}

\author{
Catherine Greenhill\thanks{Research supported by the Australian Research Council.}\\
\small School of Mathematics and Statistics\\[-0.5ex]
\small The University of New South Wales\\[-0.5ex]
\small Sydney NSW 2052, Australia\\[-0.5ex]
\small \tt c.greenhill@unsw.edu.au\\
\and
Matthew Kwan\thanks{Research supported by a UNSW Faculty of Science Summer Vacation Research Scholarship.}\\
\small School of Mathematics and Statistics\\[-0.5ex]
\small The University of New South Wales\\[-0.5ex]
\small Sydney NSW 2052, Australia\\[-0.5ex]
\small \tt matthew.a.kwan@gmail.com \\
\and
David Wind\\
\small DTU Compute\\[-0.5ex]
\small Technical University of Denmark\\[-0.5ex]
\small DK-2800 Kongens Lyngby, Denmark \\[-0.5ex]
\small \tt utdiscant@gmail.com
}

\date{18 February 2014}

\maketitle

\begin{abstract}
\global\long\def\k{d}
Let $\k\geq 3$ be a fixed integer.   We give an asympotic formula
for the expected
number of spanning trees in a uniformly random $\k$-regular graph
with $n$ vertices.
(The asymptotics are as $n\to\infty$, restricted to even $n$ if $\k$ is odd.)  
We also obtain the asymptotic distribution of the number of spanning trees
in a uniformly random cubic graph, and conjecture that the corresponding
result holds for arbitrary (fixed) $\k$. 
Numerical evidence is presented which supports our conjecture.
%
\end{abstract}

\section{Introduction}\label{s:intro}

\global\long\def\wea{\zeta}

In this paper, $\k$ denotes a fixed integer which is at least 2 (and usually
at least 3). All asymptotics are taken as $n\to\infty$, with $n$ restricted to 
even integers when $\k$ is odd.

The number of spanning trees in a graph, also called the \emph{complexity}
of the graph,
is of interest for a number of reasons.  The complexity of a graph
is an evaluation of
the Tutte polynomial (see for example~\cite{welsh}).  
The Merino-Welsh conjecture~\cite{MW} relates
the complexity of a graph with two other graph
parameters, namely, the number of acyclic orientations and the number of
totally cyclic orientations of a graph. 
(Noble and Royle~\cite{NR} recently proved that the Merino-Welsh conjecture
is true for series-parallel graphs.) The complexity of a graph also plays
a role in the theory of electrical networks (see for example~\cite{myers}).

We are interested in the number of spanning trees in random regular graphs.
The first significant result in this area is due to McKay~\cite{McKay-0}, 
who proved that for $\k\ge3$, the
$n$th root of the number of spanning trees of a random $\k$-regular
graph with $n$ vertices converges to 
\begin{equation}
\frac{\left(\k-1\right)^{\k-1}}{\left(\k^{2}-2\k\right)^{\k/2-1}}\label{eq:mckay}
\end{equation}
as $n\to\infty$, with probability one. An alternative proof of this
was later given
by Lyons~\cite[Example 3.16]{Lyons}.

McKay~\cite[Theorem 4.2]{McKay} gave an asymptotic expression for the
expected
number of spanning trees in a random graph with specified degrees,
up to some unknown constant.  His result holds when the maximum degree
is bounded 
and the average degree is bounded away from 2 (independently of $n$).
When specialised to regular degree sequences,~\cite[Theorem 4.2]{McKay}
states that the
expected number of spanning trees in $\G$ is asymptotic to 
\begin{equation}
\frac{c_{\k}}{n}\left(\frac{\left(\k-1\right)^{\k-1}}{\left(\k^{2}-2\k\right)^{\k/2-1}}\right)^{n},\label{eq:cd}
\end{equation}
for some unknown constant $c_{\k}$.

Other work on asymptotics for the number of spanning trees has focussed on
circulant graphs, grid graphs and tori (see for example~\cite{GYZ} and the
references therein).

Our first result, Theorem~\ref{thm:expectation}, provides the value of the 
constant $c_{\k}$ from (\ref{eq:cd}), proving that
\[
c_{\k}=
  \exp\left(\frac{6\k^{2}-14\k+7}{4\left(\k-1\right)^{2}}\right)\,
\frac{\left(\k-1\right)^{1/2}}{\left(\k-2\right)^{3/2}}.
\]
For our second result we investigate
the distribution of the number of spanning trees in random $\k$-regular
graphs using the small subgraph conditioning method, and obtain
the asymptotic distribution in the case of cubic graphs, presented in
Theorem~\ref{thm:distribution-3}.  We provide partial calculations for 
arbitrary fixed degrees, which lead us to conjecture that the corresponding
result holds in general (see Conjecture~\ref{cnj:distribution}).

In order to precisely state our main results we must introduce some notation
and terminology.

\subsection{Notation and our main results}\label{ss:notation}

Let $\mathbb{N}$ denote the natural numbers (which includes 0).
For integers $n$, $k$ let $(n)_{k}$ denote the falling factorial
$n(n-1)\cdots (n-k+1)$. Square brackets without subscripts denote
extraction of coefficients of a generating function.
We use $\one\left(\cdot \right)$ to denote both the indicator variable of
an event and the characteristic function of a set (the particular set will
appear as a subscript). We use standard asymptotic notation throughout,
with the exception that $\darrow$ indicates convergence in
distribution of a sequence of random variables.
(We use this notation rather than $\stackrel{d}{\rightarrow}$
to avoid overloading the symbol ``$d$'',
which we use for the degree of the graph.)

Let $\G$ denote the uniform model of $\k$-regular simple graphs 
on the vertex set $\{ 1,\ldots, n\}$. 
Define the random variable $Y_{\mathcal{G}}$ to be
the number of spanning trees in a random $G\in\G$. 

Clearly $Y_{\mathcal{G}}$ is identically zero if $n\geq 3$ and
$\k < 2$.  
A 2-regular graph has a spanning tree if and only if
it is connected (that is, forms a Hamilton cycle), in which case it has exactly $n$ spanning
trees. Hence the distribution of $Y_{\mathcal{G}}$
can be inferred from~\cite[Equation (11)]{Wormald}.
For the remainder of the paper we assume that $\k\geq 3$.

Our first result gives an asymptotic expression for the expectation
of $Y_{\mathcal{G}}$.

\begin{theorem}
\label{thm:expectation}
Let $\k\geq 3$ be a fixed integer.  Then
\[
\E Y_{\mathcal{G}} \sim
    \exp\left(\frac{6\k^{2}-14\k+7}{4\left(\k-1\right)^{2}}\right)\,
\frac{\left(\k-1\right)^{1/2}}{n\left(\k-2\right)^{3/2}}\,
\left(\frac{\left(\k-1\right)^{\k-1}}{\left(\k^{2}-2\k\right)^{\k/2-1}}
   \right)^{n}\ .
\]
\end{theorem}

This theorem is proved at the end of Section~\ref{sec:E_YX}.

Next, for fixed $\k\geq 3$ and for each positive integer $\i$, define
\begin{equation}
\lambda_{\i}\left(\k\right)  =  \frac{\left(\k-1\right)^{\i}}{2\i},\qquad
\wea_{\i}\left(\k\right)  = 
  -\frac{2\left(\k-1\right)^{\i}-1}{\left(\k-1\right)^{2\i}}
\label{eq:janson-parameters}
\end{equation}

Our second theorem gives the asymptotic distribution of the number of
spanning trees in the case of cubic graphs.

\begin{theorem}
\label{thm:distribution-3}
Let $Z_{\i}\sim\mathrm{Poisson}\left(\lambda_{\i}\left(3\right)\right)$,
with each $Z_{\i}$ independent. Consider the number of spanning trees
in a random cubic graph, normalized by the expectation given in Theorem
\ref{thm:expectation} for $\k=3$. The asymptotic distribution of
this quantity is given by
\[
\frac{Y_\mathcal{G}}{\E Y_{\mathcal{G}}} \,\,
\darrow\,\,
\prod_{\i=3}^{\infty}\, \left(1+\wea_{\i}\left(3\right)\right)^{Z_{\i}}
\, e^{-\lambda_{\i}\left(3\right)\wea_{\i}\left(3\right)}.
\]
\end{theorem}

This theorem is proved in Section~\ref{sec:EY2}.
We conjecture that an analogous result holds for arbitrary (fixed) $\k \geq 3$.

\begin{conjecture}
\label{cnj:distribution}
Let $\k\geq 3$ be fixed.  Then
\[
\frac{Y_\mathcal{G}}{\E Y_{\mathcal{G}}} \,\,
\darrow
\,\, \prod_{\i=3}^{\infty}\,
   \left(1+\wea_{\i}\left(\k\right)\right)^{Z_{\i}}\,
   e^{-\lambda_{\i}\left(\k\right)\wea_{\i}\left(\k\right)},
\]
where $Z_{\i}\sim\mathrm{Poisson}\left(\lambda_{\i}\left(\k\right)\right)$
and each $Z_{\i}$ is independent.
\end{conjecture}

We present numerical evidence which supports this conjecture in
Section~\ref{ss:numerical}.

\subsection{Plan of attack}\label{ss:plan}

From now on, we omit explicit mention
of $\k$ in the constants $\wea_{\i}=\wea_{\i}\left(\k\right)$ and
$\lambda_{\i}=\lambda_{\i}\left(\k\right)$ from (\ref{eq:janson-parameters}).

As is standard in this area, most of our calculations
will be performed in the uniform probability space $\P$ of
\textit{pairings} (also called the 
\textit{configuration model}~\cite{Bollobas,McKay-pairings,Wormald}.
Let $\k$ and $n$ be positive integers such that $\k n$ is even.
Consider a set of $\k n$ \emph{prevertices} distributed evenly into 
$n$ sets, called
\emph{buckets}.  
(We prefer the terminology ``prevertices'' to ``points''.)  
A \emph{pairing} is a partition of the prevertices into
$\k n/2$ sets of size 2, called \emph{pairs}.  
Then
\begin{equation}
|\P| = \nP\left(\k n\right) = \frac{(\k n)!}{(\k n/2)!\, 2^{\k n/2}} 
  \sim\sqrt{2}\left(\frac{\k n}{e}\right)^{\k n/2},
\label{eq:num-pairings-asymptotic}
\end{equation}
using Stirling's formula.

By contracting the prevertices in each bucket to a vertex, each pairing
projects to a labelled $\k$-regular multigraph, with loops permitted.
Let $\multiG$ denote the set of such multigraphs, and denote the projection
of a pairing $P$ by $G\left(P\right)$.
(We will occasionally informally refer to ``partial'' pairings, where
only a subset of the prevertices are paired. The projection of a partial
pairing is defined in the same way.)

Each $G\in\G$ is the projection of $\left(\k!\right)^{n}$ different
pairings (permuting the prevertices in each bucket), so we can recover
the uniform model $\G$ from $\P$ by conditioning on the event that
the projected multigraph of a random pairing is simple.

We will apply the small subgraph conditioning method in the
form given by Janson~\cite[Theorem 1]{Janson}.

\begin{theorem}
\label{thm:janson}Let $\lambda_{\i}>0$ and $\wea_{\i}\ge-1$, $\i=1,2,\dots,$
be constants. Suppose that for each $n$ we have a sequence 
$\boldsymbol{X} =(X_1,X_2,X_3,\ldots)$ of
non-negative integer valued random variables and a random variable
$Y$ with $\E Y\ne0$ (at least for large $n$). Further suppose the
following conditions are satisfied:

\begin{enumerate}[label=\textup{(A\arabic*)}]

\item\label{cond:A1}
For $m\ge1$, $\left(X_{1},\dots,X_{m}\right)\darrow\left(Z_{1},\dots,Z_{m}\right)$,
where $Z_{\i}$ are independent Poisson random variables with means
$\lambda_{\i}$; 

\item\label{cond:A2}
For any $m\ge0$, $\r\in\mathbb{N}^{m}$,
\[
\frac{\E\left[Y|X_{1}=\r_{1},\dots,X_{m}=\r_{m}\right]}{\E Y} \longrightarrow \prod_{\i=1}^{m}\left(1+\wea_{\i}\right)^{\r_{\i}}e^{-\lambda_{\i}\wea_{\i}};
\]

\item\label{cond:A3}
${\displaystyle \sum_{\i=1}^{\infty}\lambda_{\i}\wea_{\i}^{2}<\infty}$;

\item\label{cond:A4}
$
{\displaystyle \frac{\E Y^{2}}{\left(\E Y\right)^{2}}\rightarrow\exp\left(\sum_{\i=1}^{\infty}\lambda_{\i}\wea_{\i}^{2}\right)}$.\end{enumerate}
Then
\begin{equation}
\frac{Y}{\E Y}\darrow W:=\prod_{\i=1}^{\infty}\left(1+\wea_{\i}\right)^{Z_{\i}}e^{-\lambda_{\i}\wea_{\i}}.\label{eq:W}
\end{equation}
Moreover, this and the convergence in \ref{cond:A1} hold jointly.
\end{theorem}

We also need a related lemma:

\begin{lemma}
\label{lem:janson}\textup{(\cite[Lemma 1]{Janson})} Let $\lambda_{\i}'>0$,
$\i=1,2,\dots,$ be constants. Suppose that \ref{cond:A1} holds and
that $Y\ge0$.\textup{ S}uppose:

\begin{enumerate}[label=\textup{(A2$'$)}]

\item\label{cond:A2'}
For any $m\ge0$, $\r\in\mathbb{N}^{m}$,
\[
\frac{\E\left[Y\prod_{\i=1}^{m}\falling{X_{\i,n}}{\r_{\i}}\right]}{\E Y}\longrightarrow\prod_{\i=1}^{m}\left(\lambda_{\i}'\right)^{\r_{\i}}.
\]
\end{enumerate}
Then \ref{cond:A2} holds with $\lambda_{\i}\left(1+\wea_{\i}\right)=\lambda_{\i}'$
for all positive integers $\i$.
\end{lemma}

We now define the random variables $X_{\i}$ and $Y$ to which these
results will be applied.

For each $\i\geq 1$, let $\gamma_{\i}:\multiG\to\mathbb{N}$ give the number
of cycles of length $\i$ in a multigraph. (A loop is a 1-cycle, and a pair
of edges on the same two vertices is a 2-cycle.)
Then the random variable $X_{\i}=\gamma_{\i}\circ G$
is the number of $\i$-cycles in the projection of a random pairing $P\in\P$. Write
$\boldsymbol{X} = \left(X_{\i}\right)_{\i\geq 1}$ for the
sequence of all cycle counts.
It is well known~\cite{Bollobas} that for any positive integer $m$,
the random variables $X_{1},\ldots,X_{m}$ 
are asymptotically independent Poisson random variables, and that
the mean of $X_{\i}$ tends to the quantity $\lambda_{\i} = \lambda_{\i}(\k)$ 
given in (\ref{eq:janson-parameters}).
Hence Condition~\ref{cond:A1} of Theorem~\ref{thm:janson} holds.

Let $\tau:\multiG\to\mathbb{N}$ be the function which counts 
spanning trees in $\k$-regular multigraphs. Define $Y_{\mathcal{G}}$
as the restriction of $\tau$ to $\G$, and define $\Y=\tau\circ G$.
Then, $Y_{\mathcal{G}}$ is the number of spanning trees in a random
$G\in\G$, as in Section~\ref{ss:notation}. $\Y$ is accordingly
the number of spanning trees in the projection of a random pairing 
$P\in\P$. We will
investigate the asymptotic distribution of $Y_{\mathcal{G}}$
through analysis of $\Y$. 

In Section~\ref{sec:EY} we obtain an asymptotic formula for the expected
value of $Y$.  In Section~\ref{sec:E_YX} we analyse the interaction of the
number of spanning trees with short cycles, establishing that \ref{cond:A2}
holds for $\lambda_{\i}$ and $\zeta_{\i}$ as given in (\ref{eq:janson-parameters}).
This enables us to prove Theorem~\ref{thm:expectation} and to prove
that \ref{cond:A3} holds.  Then in Section~\ref{sec:EY2} we investigate
the second moment of $Y$.  We can prove that Condition~\ref{cond:A4} holds
when $\k=3$, leading to a proof of Theorem~\ref{thm:distribution-3}.
Using our partial calculations for general degrees, we provide numerical
evidence that strongly supports Conjecture~\ref{cnj:distribution}.


\section{Expected number of spanning trees }\label{sec:EY}

In this section we compute $\E Y$. Let $\T$ denote the set of labelled
trees on $n$ vertices, so that $\left|\T\right|=n^{n-2}$ by Cayley's
formula (see for example, \cite{Moon}).

Recalling the definition of $\Y$, 
we have
\[
\E\Y = \sum_{P\in\P}\frac{1}{\left|\P\right|}\,\tau\left(G\left(P\right)\right),
\]
and hence
\begin{equation}
\left|\P\right|\:\E\Y = \sum_{P\in\P}\sum_{T\in\T}M_{T,P},\label{eq:|P|EY-sum-0}
\end{equation}
where $M_{T,P}$ is the number of ways to embed $T$ into the multigraph
$G\left(P\right)$. (When $G\left(P\right)$ is simple, $M_{T,P}$
is zero or one).

Now, we want to condition on the degree of each of the $n$ vertices
in $T$. Define the set of possible degree sequences 
\[
\D=\left\{ \d\in\mathbb{N}^{n}:\quad\sum_{\i=1}^{n}\d_{\i}=2\left(n-1\right)\right\} .
\]
We can decompose $\left|\P\right|\:\E\Y$ as
\begin{equation}
\sum_{\d\in\D}\sum_{T\in\T}\one\left(T\sim\d\right)\,
  \sum_{P\in\P}\, M_{T,P},\label{eq:|P|EY-sum-1}
\end{equation}
where $T\sim\d$ denotes the event that vertex $\i$ has degree $\d_{\i}$
in $T$, for all $\i=1,\ldots n$.

To evaluate the innermost sum in (\ref{eq:|P|EY-sum-1}), fix some
$\d\in\D$ and some $T\in\T$ with $T\sim\d$. We need to count the
number of pairings that include $T$, with the embedding of $T$ identified.
That is, if for some pairing $P$, the tree $T$ can be embedded in
$G\left(P\right)$ in multiple ways, then we count each different
way separately.

Now, exactly $\d_{\i}$ of the prevertices in bucket $\i$ must contribute
to $T$, and there are $\falling{\k}{\d_{\i}}$ ways to choose and
order these prevertices. So, there are
\[
\prod_{\i=1}^{n}\falling{\k}{\d_{\i}}
\]
ways to pair up the $n-1$ edges corresponding to a copy of $T$.
Then, there are 
\[
\k n-2\left(n-1\right)=\left(\k-2\right)n+2
\]
prevertices remaining, which can be paired in $\nP\left(\left(\k-2\right)n+2\right)$
ways. This yields
\begin{equation}
\left|\P\right|\,\E\Y=
  \nP\left(\left(\k-2\right)n+2\right)\,
 \sum_{\d\in\D}\prod_{\i=1}^{n}\, \falling{\k}{\d_{\i}} \,
 \sum_{T\in\T}
   \, \one\left(T\sim\d\right).
\label{eq:|P|EY-sum-2}
\end{equation}
The inner sum in (\ref{eq:|P|EY-sum-2}) is the number of trees with
degree sequence $\d$, which is the multinomial
\begin{equation}
\binom{n-2}{\d_{1}-1,\ldots,\d_{n}-1}
  =\frac{\left(n-2\right)!}{\prod_{\i=1}^{n}\left(\d_{\i}-1\right)!}.
\label{eq:moon}
\end{equation}
(See for example Moon~\cite[Theorem 3.1]{Moon}.)

Hence
\[
\left|\P\right|\:\E\Y = \left(n-2\right)!\,\nP\left(n\k-2\left(n-1\right)\right)\sum_{\d\in\D}\prod_{\i=1}^{n}\frac{\falling{\k}{\d_{\i}}}{\left(\d_{\i}-1\right)!}.
\]
It follows that the total number of ways to choose a spanning tree on $n$
vertices and choose a partial pairing that projects to that tree is 
\begin{align}
(n-2)!\, \sum_{\d\in\D}\, \prod_{\i=1}^{n}\, 
  \frac{\falling{\k}{\d_{\i}}}{\left(\d_{\i}-1\right)!} 
  &=  (n-2)!\, \left[x^{2\left(n-1\right)}\right]\,
  \left(\sum_{\i=1}^{\infty}\, 
   \frac{\falling{\k}{\i}}{\left(\i-1\right)!}\, x^{\i}\right)^{n}\nonumber \\
 &=  (n-2)!\, \left[x^{2\left(n-1\right)}\right]\, 
   \left(\k x\left(1+x\right)^{\k-1}\right)^{n}\nonumber\\
 &=  (n-2)!\, \k^{n}\, \binom{(\k-1)n}{n-2}. \label{useful}
\end{align}
Hence, by Stirling's approximation and (\ref{eq:num-pairings-asymptotic})
we conclude that
\begin{align}
\left|\P\right|\,\E Y & =  \left(n-2\right)!\,\nP\left(n\k-2\left(n-1\right)\right)\,\k^{n}\, \binom{(\k-1)n}{n-2}\nonumber \\
 & \sim  \frac{\sqrt{2}\left(\k-1\right)^{1/2}}{n\left(\k-2\right)^{3/2}}\left(\k\left(\k-2\right)\left(\k-1\right)^{\k-1}\left(\frac{n}{\left(\k-2\right)e}\right)^{\k/2}\right)^{n}.\label{eq:|P|EY-asymptotic}
\end{align}

It follows that
\begin{align}
\E Y & = \frac{\left(n-2\right)!\,\nP\left(n\k-2\left(n-1\right)\right)}{\nP\left(n\k\right)}\,\k^{n}\, \binom{(\k-1)n}{n-2} \nonumber \\
 & \sim \frac{\left(\k-1\right)^{1/2}}{n\left(\k-2\right)^{3/2}}\left(\frac{\left(\k-1\right)^{\k-1}}{\left(\k^{2}-2\k\right)^{\k/2-1}}\right)^{n}.
\label{eq:EY-asymptotic}
\end{align}
Hence for $\k\ge3$ and $n$ sufficiently large, we have $\E\Y\ne0$.

\section{Interaction with short cycles}\label{sec:E_YX}

\global\long\def\Xr{X_{\r}}
\global\long\def\newP{\mathcal{P}_{n,\r}}
\global\long\def\s{s}
\global\long\def\nC{n_{\r}}
\global\long\def\NC{\left|\r\right|}
\global\long\def\dI{\delta'}
\global\long\def\DI{\mathcal{D}_{Q}}
\global\long\def\nq{t}
\global\long\def\W{E}
\global\long\def\Dd{\mathcal{D}_{\dI}}
\global\long\def\z{\mu}
\global\long\def\zz{\kappa}
\global\long\def\GF{\Lambda}
\global\long\def\gf{g}
\global\long\def\R{\mathcal{R}_{n,\r}}
\global\long\def\Q{\mathcal{I}_{\i}}
\global\long\def\Qr{\mathcal{Q}_{\r}}

Recall that $X_{\i}=\gamma_{\i}\circ G$ is the number of cycles of
length $\i$ in the projection of a random pairing $P\in\P$.
For some fixed $m\ge0$, $\r\in\mathbb{N}^{m}$, let $\Xr=\prod_{\i=1}^{m}\falling{X_{\i}}{\r_{\i}}$.
In this section we will compute an asymptotic formula for $\E\left[\Y\Xr\right]/\E\Y$,
in the form required by Condition \ref{cond:A2'}.

We have
\[
\E\left[\Y\Xr\right]=\sum_{P\in\P}\frac{1}{\left|\P\right|}\,\tau\left(G\left(P\right)\right)\prod_{\i=1}^{m}\falling{\gamma_{\i}\left(G\left(P\right)\right)}{\r_{\i}}.
\]
Note that
\[
\prod_{\i=1}^{m}\falling{\gamma_{\i}\left(G\left(P\right)\right)}{\r_{\i}}
\]
is the number of ways to choose, for each $\i\in\{ 1,\ldots, m\}$,
 an ordered set of $\r_{\i}$ cycles
of length $\i$. 
This will result in an ordered set of 
\[
\NC=\sum_{\i=1}^{m}\r_{\i}
\]
cycles.

We make the decomposition 
\[
\prod_{\i=1}^{m}\falling{\gamma_{\i}\left(G\left(P\right)\right)}{\r_{\i}}
  =\gamma_{\r}^{\left(0\right)}+\gamma_{\r}',
\]
where $\gamma_{\r}^{\left(0\right)}$ is the number of ordered sets of cycles
in which each cycle is disjoint, and $\gamma_{\r}'$ is the number
of ordered sets in which some vertices are shared between multiple cycles.
We can further decompose $\gamma_{\r}'$ by the structure of the interaction
between the cycles. That is, according to the multigraph that is the
union of the cycles, and the specification of which edges of this
union belong to which cycle. This expresses $\gamma_{\r}'$ as a sum
of terms $\gamma_{\r}^{\left(\i\right)}$. The number of terms $J$
in this decomposition depends on $\rho$, but is $O\left(1\right)$
as $n\to\infty$.

Define 
\[
\W^{\left(\i\right)}=\sum_{P\in\P}\,
  \tau\left(G\left(P\right)\right)\, \gamma_{\r}^{\left(\i\right)},
\]
so that we have $\left|\P\right|\,\E\left[\Y\Xr\right]=\sum_{\i=1}^{J}\W^{\left(\i\right)}$.

We proceed to calculate $\W^{\left(0\right)}$. As is standard when
applying this method (see for example, \cite[Theorem 9.6]{JLR}),
$\E\left[\Y\Xr\right]$ is asymptotically dominated by $\W^{\left(0\right)}$
(the contribution due to disjoint cycles). See Lemma \ref{lem:Ej}
for some justification for this fact.

Let $\mathcal{C}_{n,\i}$ be the set of all $\i$-cycles on the vertex
set $\left\{ 1,\dots,n\right\} $, and define the Cartesian product
\[
\mathcal{C}_{n,\r}=\prod_{\i=1}^{m}\mathcal{C}_{n,\i}^{\r_{\i}}.
\]
Each $R\in\mathcal{C}_{n,\r}$ is an ordered set of $\left|\r\right|$
cycles. We use the notation $R_{\i,\ii}$ for the $\ii$th cycle of
length $\i$ in $R$. Next, define
\[
\R = \left\{ R\in\mathcal{C}_{n,\r}:\mbox{ the cycles in $R$ are pairwise
 disjoint} \right\} .
\]
Similarly to (\ref{eq:|P|EY-sum-0}), we have
\[
\W^{\left(0\right)} = \sum_{R\in\R}\sum_{T\in\T}\sum_{P\in\P}M_{\left(T,R\right),P},
\]
where $M_{\left(T,R\right),P}$ is the number of ways to embed the
tree $T$ and the cycles in $R$ into the multigraph $G\left(P\right)$.

We further condition on the edge intersection between the embedding
of $T$ and the cycles in $R$. We use a binary sequence of length
$\i$ to encode the intersection of an $\i$-cycle with a spanning
tree. Picking an arbitrary start vertex and direction for the cycle,
if the $\ii$th edge of the cycle is to be included in the intersection,
then the $\ii$th element of the corresponding sequence is one; otherwise
it is zero. All sequences $q\in\left\{ 0,1\right\} ^{\i}$ represent
possible intersections, except the sequence $\left(1,\dots,1\right)$,
because a tree contains no cycles.

Define the set of all possible intersection sequences for a cycle
of length $\i$, by 
\[
\Q=\left\{ 0,1\right\} ^{\i}\setminus \left(1,\dots,1\right).
\]
Also, define the Cartesian product
\[
\Qr=\prod_{\i=1}^{m}\Q^{\r_{\i}}.
\]
So, for each $R\in \R$, specifying some $Q\in\Qr$ fully specifies the
intersection between the cycles in $R$ and a tree $T$.

We have 
\begin{equation}
\W^{\left(0\right)}=\sum_{Q\in\Qr}\sum_{R\in\R}\sum_{T\in\T}\sum_{P\in\P}M_{\left(T,R,Q\right),P},\label{eq:|P|E[YX]}
\end{equation}
where $M_{\left(T,R,Q\right),P}$ is the number of ways to embed $T$
and the cycles in $R$ in $P$, such that the intersection between
the embedding of $T$ and the cycle $R_{\i,\ii}$ is consistent with
$Q_{\i,\ii}$, for $\i = 1,\ldots, m$ and $\ii = 1,\ldots, \r_{\i}$.

Fixing $Q\in\Qr$, we will now evaluate the innermost triple sum in
(\ref{eq:|P|E[YX]}). Consider the following process:

\renewcommand{\theenumi}{\arabic{enumi}}
\renewcommand{\labelenumi}{\theenumi.}
\setenumerate{itemsep=5pt}
\begin{enumerate}
\item \label{step:3}Choose some $R\in\R$.
\item \label{step:4} Choose a partial pairing that projects to $R$.
\item \label{step:6} Extend this to a pairing of a spanning tree consistent
with $Q$.
\item \label{step:7}Pair the remaining prevertices arbitrarily.
\end{enumerate}

We will find that the number of ways to complete each step is independent
of the other steps. Then, $\W^{\left(0\right)}$ is a product of the
number of ways to complete each step, summed over all $Q\in\Qr$.

First let $\nC=\sum_{\i=1}^{m}\i\r_{\i}$ be the total number of vertices
in each $R\in\R$. The number of ways to choose the vertices for some
$R\in\R$ is
\[
\binom{n}{\nC}
\]
and the number of different arrangements of disjoint cycles on those
vertices is 
\[
\frac{\nC!}{\prod_{\i=1}^{m}\left(\i\chi\left(\i\right)\right)^{\r_{\i}}},
\]
where $\i\chi\left(\i\right)$ is the size of the automorphism group
of a $\i$-cycle:
\[
\chi\left(\i\right)=\begin{cases}
1 & \mbox{for }\i\le2\\
2 & \mbox{for }\i>2.
\end{cases}
\]
That is, the number of ways to complete Step \ref{step:3} is 
\[
\s_{\ref{step:3}}^{\left(0\right)}=\frac{n!}{\left(n-\nC\right)!\prod_{\i=1}^{m}\left(\i\chi\left(\i\right)\right)^{\r_{\i}}}.
\]
Next, the number of ways to complete Step \ref{step:4} is
\[
\s_{\ref{step:4}}^{\left(0\right)}=\prod_{\i=1}^{m}\left(\frac{\chi\left(\i\right)\left(\k\left(\k-1\right)\right)^{\i}}{2}\right)^{\r_{\i}}.
\]
Note for future reference that we have 
\begin{equation}
\s_{\ref{step:3}}^{\left(0\right)}\s_{\ref{step:4}}^{\left(0\right)}\sim n^{\nC}\prod_{\i=1}^{m}\left(\frac{\left(\k\left(\k-1\right)\right)^{\i}}{2\i}\right)^{\r_{\i}}.\label{eq:f_3f_4}
\end{equation}
Next, we count the number of ways to extend this pairing to
a tree $T$ consistent with $Q$. We do this by constructing a new
irregular pairing model $\newP$ from the prevertices still unpaired after
Step \ref{step:4}. Recall that $Q$ describes a union of disjoint
paths; for each of these paths, combine the unpaired prevertices remaining
in each constituent vertex of the path to form a \emph{super-bucket}. If the
path has $\ii$ vertices then the resulting super-bucket
has $\ii\left(\k-2\right)$ prevertices. Let $\left|Q\right|$ be the number
of super-buckets formed in this way, so the total number of buckets in
$\newP$ is $n':=n-\nC+\left|Q\right|$.

Now, consider an extension of a pairing of cycles from Step \ref{step:4},
to a (partial) pairing of the edges of a tree $T$ consistent with
$Q$, as per Step \ref{step:6}. The pairs from this extension correspond
uniquely to a (partial) pairing $P'$ in the pairing model $\newP$.
By the construction of $\newP$, the projection $T'=G\left(P'\right)$
of this pairing is simply $T$ with some subpaths contracted
to single vertices. Since contracting edges of a tree cannot create
cycles, $G\left(P'\right)$ is itself a (spanning) tree. Similarly,
every pairing of a tree in $\newP$ corresponds to an extension of
a pairing of cycles to a pairing of a tree in $\P$ consistent with
$Q$. So the number of ways $\s_{\ref{step:6}}^{\left(0\right)}$
to complete Step \ref{step:6} equals the number of ways to
choose and pair up a spanning tree in $\newP$.

We will perform this count as in Section \ref{sec:EY}, by conditioning
on the degree in $T'$ of each bucket in $\newP$. Put an arbitrary
ordering on the $\left|Q\right|$ super-buckets, and let $d_{\i}$
be the number of prevertices in the $\i$th super-bucket. For a degree
sequence $\d$, let $\left|\d\right|$ be its degree sum. Define the
sets
\begin{align*}
\DI & = \left\{ \dI\in\mathbb{N}^{\left|Q\right|}:\quad\dI_{\i}\le d_{\i}\mbox{ for all }\i\right\} ,\\
\Dd & = \left\{ \d\in\mathbb{N}^{n-\nC}:\quad\sum_{\i=1}^{n-\nC}\d_{\i}=2\left(n'-1\right)-\left|\dI\right|\right\} .
\end{align*}
The set $\DI$ contains all possible degree-in-$T'$ sequences for
the $\left|Q\right|$ super-buckets. For some $\dI\in\DI$, the set
$\Dd$ contains all possible degree sequences for the $n-\nC$ remaining
ordinary buckets. So, we have
\[
\s_{\ref{step:6}}^{\left(0\right)}=
  \sum_{\dI\in\DI} \sum_{\d\in\Dd}\, \sum_{T'\in\mathcal{T}_{n'}}\,
  \one\left(T'\sim\left(\dI,\d\right)\right)\, \sum_{P'\in\newP}M_{T',P},
\]
where $T'\sim\left(\dI,\d\right)$ denotes the event that the super-buckets
have degree-in-$T'$ sequence $\dI$ and the remaining vertices have
degree-in-$T'$ sequence $\d$. Proceeding as before, after fixing
some $\left(T',\dI,\d\right)$, there are
\[
\left(\prod_{\i=1}^{n-\nC}\falling \k{\d_{\i}}\right)\left(\prod_{\i=1}^{\left|Q\right|}\falling{d_{\i}}{\dI_{\i}}\right)
\]
ways to pair the edges of $T'$. There are 
\[
\frac{\left(n'-2\right)!}{\left(\prod_{\i=1}^{n-\nC}\left(\d_{\i}-1\right)!\right)\left(\prod_{\i=1}^{\left|Q\right|}\left(\dI_{\i}-1\right)!\right)}
\]
trees with $T'\sim\left(\dI,\d\right)$, by (\ref{eq:moon}). So,
we have
\begin{equation}
\s_{\ref{step:6}}^{\left(0\right)}=\sum_{\dI\in\DI}A_{\dI}\prod_{\i=1}^{\left|Q\right|}\frac{\falling{d_{\i}}{\dI_{\i}}}{\left(\dI_{\i}-1\right)!},\label{eq:f_6}
\end{equation}
where
\begin{align}
A_{\dI} & = \left(n'-2\right)!\sum_{\d\in\Dd}\prod_{\i=1}^{n-\nC}\frac{\falling \k{\d_{\i}}}{\left(\d_{\i}-1\right)!}\nonumber \\
 & = \left(n'-2\right)!\,\left[x^{2\left(n'-1\right)-\left|\dI\right|}\right]\left(\k x\left(1+x\right)^{\k-1}\right)^{n-\nC}\nonumber \\
 & = \left(n'-2\right)!\,\k^{n-\nC}\,
  \binom{(\k-1)(n-\nC)}{2(n'-1)-|\dI|-(n-\nC)}\nonumber \\
 & \sim \frac{\left(\k-2\right)^{2\left|Q\right|-\left|\dI\right|-5/2}\,\left(\k-1\right)^{1/2}}{n^{\nC-\left|Q\right|+2}}\left(\frac{\left(\k-2\right)^{\k-2}}{\k\left(\k-1\right)^{\k-1}}\right)^{\nC}\left(\frac{\k\left(\k-1\right)^{\k-1}n}{e\left(\k-2\right)^{\k-2}}\right)^{n}.\label{eq:A_d'-asymptotic}
\end{align}
Finally, for Step \ref{step:7} there are $\k n-2\nC-2\left(n'-1\right)=\left(\k-2\right)n-2\left(\left|Q\right|-1\right)$
prevertices remaining, which can be paired in 
\begin{equation}
\s_{\ref{step:7}}^{\left(0\right)}=\nP\left(\left(\k-2\right)n-2\left(\left|Q\right|-1\right)\right)\sim\sqrt{2}\left(\frac{\left(\k-2\right)n}{e}\right)^{\left(\k/2-1\right)n}\left(\left(\k-2\right)n\right)^{1-\left|Q\right|}\label{eq:f_7}
\end{equation}
 ways.

Combining (\ref{eq:|P|EY-asymptotic}), (\ref{eq:f_3f_4}), (\ref{eq:f_6}),
(\ref{eq:A_d'-asymptotic}) and (\ref{eq:f_7}), we have
\begin{align}
\frac{\W^{\left(0\right)}}{\left|\P\right|\,\E\Y} 
 & = \sum_{Q\in\Qr}\, 
  \frac{\s_{\ref{step:3}}^{\left(0\right)}\s_{\ref{step:4}}^{\left(0\right)}\s_{\ref{step:6}}^{\left(0\right)}\s_{\ref{step:7}}^{\left(0\right)}}
 {\left|\P\right|\,\E\Y}\nonumber \\
 & \to \left(\prod_{\i=1}^{m}\left(\frac{\left(\k\left(\k-1\right)\right)^{\i}}{2\i}\right)^{\r_{\i}}\right)\nonumber \\
 &  \quad\times \sum_{Q\in \Qr}\, \sum_{\dI\in\DI}\left(\k-2\right)^{\left|Q\right|-\left|\dI\right|}\left(\frac{\left(\k-2\right)^{\k-2}}{\k\left(\k-1\right)^{\k-1}}\right)^{\nC}\prod_{\i=1}^{\left|Q\right|}\frac{\falling{d_{\i}}{\dI_{\i}}}{\left(\dI_{\i}-1\right)!}.\label{eq:E0}
\end{align}

As is standard in these arguments, the only significant contribution to
$\E\left[\Y \Xr\right]$ comes from $\W^{\left(0\right)}$, where the cycles
do not overlap.  For completeness we sketch a proof of this below.

\begin{lemma}
\label{lem:Ej}$\E\left[\Y\Xr\right]$ is dominated by the contribution
from $\W^{\left(0\right)}$. That is, 
\[
\frac{\E\left[\Y\Xr\right]}{\E\Y}\sim\frac{\W^{\left(0\right)}}{\left|\P\right|\,\E\Y}.
\]
\end{lemma}

\begin{proof}
We can estimate general $\W^{\left(\i\right)}$ with some slight modifications
to the above calculations. We would need a different $\R'\subseteq\mathcal{C}_{n,\r}$
that contains all possible ways to embed an ordered set of cycles
with a particular union $U$ into the vertex set $\left\{ 1,\dots,n\right\} $.
The intersection between a spanning tree and the cycles in some $R\in\R$
would then be a subforest of $U$, so it would be more complicated
to explicitly define a set $\Qr$ that encodes all possibilities for
the intersection. However, the number of possible intersections is
still independent of $n$.

Using the same 4 steps, the decomposition $\W^{\left(\i\right)}=\s_{\ref{step:3}}^{\left(\i\right)}\s_{\ref{step:4}}^{\left(\i\right)}\s_{\ref{step:6}}^{\left(\i\right)}\s_{\ref{step:7}}^{\left(\i\right)}$
is still valid. Let $\left|U\right|$ and $\left\Vert U\right\Vert $
be the number of vertices and edges in the multigraph $U$, respectively.
Carefully adjusting the calculations for $\W^{\left(0\right)}$, we
have $\s_{\ref{step:3}}^{\left(\i\right)}/\s_{\ref{step:3}}^{\left(0\right)}=O\left(n^{\left|U\right|-\nC}\right)$,
$\s_{\ref{step:4}}^{\left(\i\right)}/\s_{\ref{step:4}}^{\left(0\right)}=O\left(1\right)$
and $\s_{\ref{step:6}}^{\left(\i\right)}/\s_{\ref{step:6}}^{\left(0\right)}=O\left(1\right)$.
For Step \ref{step:7} there would be $\k n-2\left\Vert U\right\Vert -2\left(n'-1\right)$
prevertices remaining, so $\s_{\ref{step:7}}^{\left(\i\right)}/\s_{\ref{step:7}}^{\left(0\right)}=O\left(n^{\nC-\left\Vert U\right\Vert }\right)$.

We conclude that $\W^{\left(\i\right)}/\W^{\left(0\right)}=O\left(n^{\left|U\right|-\left\Vert U\right\Vert }\right)$.
Any non-disjoint union between distinct cycles has more edges than
vertices, so the lemma is proved.
\end{proof}
We now want to to express (\ref{eq:E0}) in the form required by \ref{cond:A2'}.
So, we consider each cycle independently. For a sequence $q\in\Q$,
let $q\left[\ii\right]$ be the number of paths with $\ii$ vertices
in the intersection encoded by $q$. So, for $Q\in\Qr$ we have 
\[
\left|Q\right|=\sum_{\i=1}^{m}\sum_{\iii=1}^{\r_{\i}}\sum_{\ii=1}^{\i}Q_{\i,\iii}\left[\ii\right].
\]
Also, let $\left|q\right|=\sum_{\ii=1}^{\i}q\left[\ii\right]$ be
the total number of paths in the intersection encoded by $q$.

Now, recall that if the $\i$th super-bucket was collapsed from a
path of length $\ii$, then $d_{\i}=\ii\left(\k-2\right)$. Also recall
that $\nC=\sum_{\i=1}^{m}\i\r_{\i}$ and note that 
\[
\left(\k-2\right)^{\left|Q\right|-\left|\dI\right|}=\prod_{\i=1}^{\left|Q\right|}\left(\k-2\right)^{1-\dI_{\i}}.
\]
As a result, we have
\[
\frac{\E\left[\Y\Xr\right]}{\E\Y}\rightarrow \prod_{\i=1}^{m}\left(\lambda_{\i}'\right)^{\r_{\i}},
\]
where
\[
\lambda_{\i}' = \frac{\left(\k\left(\k-1\right)\right)^{\i}}{2\i}\left(\frac{\left(\k-2\right)^{\k-2}}{\k\left(\k-1\right)^{\k-1}}\right)^{\i}\sum_{q\in\Q}\prod_{\ii=1}^{\i}\left(\sum_{\iii=1}^{\ii\left(\k-2\right)}\frac{\falling{\ii\left(\k-2\right)}{\iii}}{\left(\iii-1\right)!\,\left(\k-2\right)^{\iii-1}}\right)^{q\left[\ii\right]}.
\]
Note that $\iii$ takes the role of $\dI_{\i}$ for the $\i$th super-bucket.
We have proved that Condition \ref{cond:A2'} is satisfied. It remains
to simplify our expression for $\lambda_{\i}'$. We have
\begin{align*}
\lambda_{\i}' & = \frac{1}{2\i}\,
  \left(\frac{\k-2}{\k-1}\right)^{\i\left(\k-2\right)}\,
  \sum_{q\in\Q}\, \prod_{\ii=1}^{\i}\, \left(\ii\left(\k-2\right)\,
  \sum_{\iii=0}^{\ii\left(\k-2\right)-1}\,
  \binom{\ii(\k-2)-1}{\iii}\, 
  \left(\k-2\right)^{-\iii}\right)^{q\left[\ii\right]}\\
 & = \frac{1}{2\i}\left(\frac{\k-2}{\k-1}\right)^{\i\left(\k-2\right)}\sum_{q\in\Q}\prod_{\ii=1}^{\i}\left(\ii\left(\k-2\right)\left(\frac{1}{\k-2}+1\right)^{\ii\left(\k-2\right)-1}\right)^{q\left[\ii\right]}\\
 & = \frac{1}{2\i}\sum_{q\in\Q}\left(\frac{\left(\k-2\right)^{2}}{\k-1}\right)^{\left|q\right|}\prod_{\ii=1}^{\i}\ii^{q\left[\ii\right]}.
\end{align*}
Now, recall that $\left(1,\dots,1\right)\notin\Q$. So, to evaluate
the sum over $q$, we may identify a particular element in the sequence
to be zero. By symmetry, we arbitrarily choose the last. We also condition
on $\left|q\right|$: define
\begin{equation}
\z  =  \frac{\left(\k-2\right)^{2}}{\k-1},\qquad
\GF_{\i,\nq}  =  \sum_{\substack{q\in\Q\\
\left|q\right|=\nq\\
q_{\i}=0
}
}\z^{\nq}\prod_{\ii=1}^{\i}\ii^{q\left[\ii\right]}.
\end{equation}
Note that $\GF_{\i,1}=\i\z$ for all $\i$, because the only sequence
$q\in\Q$ with $\left|q\right|=1$ and $q_{\i}=0$ is $\left(1,\dots,1,0\right)$.
For $\left|q\right|>1$, the first path in (the intersection encoded
by) a sequence can contain anywhere between 1 and $\i-1$ vertices.
Ranging over the possibilities, we have
\[
\GF_{\i,\nq}=\sum_{\ii=1}^{\i-1}\ii\z\GF_{\i-\ii,\nq-1}
\]
for $\nq>0$. To solve this recurrence, define the generating function
\[
\GF\left(x,y\right) = \sum_{\i=1}^{\infty}\sum_{\nq=1}^{\infty}\GF_{\i,\nq}x^{\i}y^{\nq}.
\]
We have
\[
\GF\left(x,y\right)-\sum_{\i=1}^{\infty}\GF_{\i,1}x^{\i}y=\sum_{\i=1}^{\infty}\sum_{\nq=2}^{\infty}\sum_{\ii=1}^{\i-1}\ii\z\GF_{\i-\ii,\nq-1}x^{\i}y^{\nq},
\]
so
\begin{align*}
\GF\left(x,y\right) & = \sum_{\ii=1}^{\infty}\ii x^{\ii}y\z\sum_{\i=\ii+1}^{\infty}\sum_{\nq'=1}^{\infty}\GF_{\i-\ii,\nq'}x^{\i-\ii}y^{\nq'}+\sum_{\i=1}^{\infty}\i x^{\i}y\z\\
 & = \left(\GF\left(x,y\right)+1\right)\sum_{\ii=1}^{\infty}\ii x^{\ii}y\z.
\end{align*}
Now, defining
\[
\gf\left(x\right)=\sum_{\ii=1}^{\infty}\ii x^{\ii}\z=\frac{x\z}{\left(1-x\right)^{2}},
\]
we have 
\[
\GF\left(x,y\right)=\frac{\gf\left(x\right)y}{1-\gf\left(x\right)y}.
\]
If $q\in\Q$, then there are $\i$ positions to place a zero, and
if $\left|q\right|=\nq$ then there are $\nq$ zeros in $q$.  Hence
\begin{align*}
\lambda_{\i}' & =  \frac{1}{2\i}\sum_{\nq=1}^{\i}\frac{\i\GF_{\i,\nq}}{\nq}\\
  &=  \frac{1}{2}\left[x^{\i}\right]\sum_{\nq=1}^{\i}\frac{1}{\nq}\left[y^{\nq-1}\right]\frac{\gf\left(x\right)}{1-\gf\left(x\right)y}\\
 & =  \frac{1}{2}\left[x^{\i}\right]\int_{0}^{1}\frac{\gf\left(x\right)}{1-\gf\left(x\right)y}\mathrm{d}y\\
 & =  -\frac{1}{2}\left[x^{\i}\right]\log\left(1-\gf\left(x\right)\right).
\end{align*}
Now, defining $\zz=\sqrt{\k-1}$ we
have
\[
1-\gf\left(x\right)=\frac{1+x^{2}-\left(2+\z\right)x}{\left(1-x\right)^{2}}=\frac{\left(1-\zz^{2}x\right)\left(1-\zz^{-2}x\right)}{\left(1-x\right)^{2}},
\]
so 
\begin{align}
\lambda_{\i}' & = \frac{1}{2}\left[x^{\i}\right]\left(2\log\left(1-x\right)-\log\left(1-\zz^{2}x\right)-\log\left(1-\zz^{-2}x\right)\right) \nonumber\\
 & = \frac{1}{2}\left[x^{\i}\right]\sum_{\ii=1}^{\infty}\frac{-2x^{\ii}+\left(\zz^{2}x\right)^{\ii}+\left(\zz^{-2}x\right)^{\ii}}{\ii}\nonumber\\
 & = \frac{1}{2\i}\left(\zz^{\i}-\zz^{-\i}\right)^{2}\nonumber\\
 & = \frac{\left(\left(\k-1\right)^{\i}-1\right)^{2}}{2\i\left(\k-1\right)^{\i}}
\label{eq:lambda'}
\end{align}
for $\i \geq 1$.

To complete this section we will 
establish that conditions \ref{cond:A2} and \ref{cond:A3} of 
Theorem~\ref{thm:janson} hold, and prove
Theorem~\ref{thm:expectation}.

\begin{lemma}
Let $\k\geq 3$ be a fixed integer.  Then 
Conditions~\ref{cond:A2} and~\ref{cond:A3} of Theorem~\ref{thm:janson} are 
satisfied, and 
\[
\exp\left(\sum_{\i=1}^{\infty}\lambda_{\i}\wea_{\i}^{2}\right)=\frac{\k^{2}}{\sqrt{\left(\k-1\right)\left(\k-2\right)\left(\k^{2}-\k+1\right)}}.
\]
\label{EY2/EY2-target}
\end{lemma}

\begin{proof}
The calculations of this section show that Condition~\ref{cond:A2'}
of Lemma~\ref{lem:janson} is satisfied with $\lambda_{\i}'$ given by
(\ref{eq:lambda'}).  Then Lemma~\ref{lem:janson} guarantees that \ref{cond:A2}
is satisfied.
Using the Taylor expansion of $\log\left(1-z\right)$,
it follows from (\ref{eq:janson-parameters}) that
\begin{align*}
& \sum_{\i=1}^{\infty}\lambda_{\i}\wea_{\i}^{2} \\
&=
\sum_{\i=1}^{\infty} \, \dfrac{1}{2\i}\left(4\left(\k-1\right)^{-\i}-4\left(\k-1\right)^{-2\i}+\left(\k-1\right)^{-3\i}\right)\\
&=\dfrac{1}{2}\left(-4\log\left(1-\left(\k-1\right)^{-1}\right)+4\log\left(1-\left(\k-1\right)^{-2}\right)-\log\left(1-\left(\k-1\right)^{-3}\right)\right).
\end{align*}
Taking the exponential of both sides and rearranging establishes
the stated expression for 
$\exp\left(\sum_{\i = 1}^\infty \lambda_{\i}\wea_{\i}^2\right)$,
which is finite for $\k\geq 3$.  Hence Condition~\ref{cond:A3} holds, 
as required.
\end{proof}

\begin{proof}[Proof of Theorem \ref{thm:expectation}]
We know that condition~\ref{cond:A2} holds (as proved above), and hence
\[
\E Y_{\mathcal{G}}  =  \E\left[Y|X_{1}=X_{2}=0\right]\\
  \rightarrow   \E Y\,\exp\left(-\lambda_{1}\wea_{1}-\lambda_{2}\wea_{2}\right).
\]
Substituting using (\ref{eq:janson-parameters}) and (\ref{eq:EY-asymptotic}) 
completes the proof.
\end{proof}

\section{The second moment}\label{sec:EY2}

\global\long\def\r{\delta^{1}}
\global\long\def\w{\delta^{2}}
\global\long\def\rr{\eta^{1}}
\global\long\def\ww{\eta^{2}}
\global\long\def\q{\eta^{3}}
\global\long\def\floor#1{\left\lfloor #1\right\rfloor }
\global\long\def\ceil#1{\left\lceil #1\right\rceil }
\global\long\def\C{C}
\global\long\def\X{\mathcal{X}}
\global\long\def\Xk#1{\X^{\left(#1\right)}}
\global\long\def\H{F_n}
\global\long\def\Hk#1{\H^{\left(#1\right)}}
\global\long\def\x{x}
\global\long\def\y{y}
\global\long\def\f{f}
\global\long\def\S{\mathcal{S}}
\global\long\def\g{a_{n}}
\global\long\def\gg{\hat{a}_{n}}
\global\long\def\circle{\mathcal{O}}

We now want to calculate $\E\Y^{2}$. 
As a first step we transform this problem into one of evaluating
the coefficient of a certain generating function.

\begin{lemma}
Let $\k\geq 3$ be fixed and define
\[ N(n,\k) = \begin{cases} n & \text{ if $\k\geq 4$,}\\
            n/2 + 2 & \text{ if $\k=3$}.
\end{cases}\]
Then
\begin{align*}
& \left|\P\right|\,\E\Y^{2} \\
&= \frac{n!\, \left((\k-2)n\right)!\, \k^n}{2^{(\k/2-1)n+2}}\,
   \sum_{b=1}^{N(n,\k)}\, \frac{2^b}
  {b!\,\left(\left(\k/2-1\right)n-b+2\right)!}\,
  \left[z^{n}\right]\,\left(\sum_{\i=1}^{\infty}
 \binom{(\k-1)\i}{\i}\, z^{\i}\right)^{b}.
\end{align*}
\label{lem:|P|EY2}
\end{lemma}

\begin{proof}
We write 
\[
\left|\P\right|\,\E\Y^{2} = \sum_{P\in\P}\, \sum_{T_{1}\in\T}\,
  \sum_{T_{2}\in\T}\, M_{\left(T_{1},T_{2}\right),P},
\]
where $M_{(T_{1},T_{2}),P}$ is the number of ways to embed the ordered
pair of trees $\left(T_{1},T_{2}\right)$ into the multigraph $G\left(P\right)$.
We will estimate this sum by choosing some $T_{1},T_{2}\in\T$ and
counting the ways to pair up their edges, then counting the ways to
complete the pairing. We break up this process in a similar way to
Section \ref{sec:E_YX}:

\begin{enumerate}
\item Choose $b\in\{1,\ldots, n\}$, which will be the number of 
connected components in the 
intersection of the embeddings of $T_{1}$ and $T_{2}$.
(As we will see later, when $\k=3$ we must restrict to $b\in \{1,\ldots,
n/2 + 2\}$.)
\item \label{step:EY2-1}
Choose a partition $(\nu_1,\ldots, \nu_b)$ of $n$ into positive parts.
That is, $\nu_{\i}$ is a positive integer for $\i=1,\ldots, b$ and
$\sum_{\i=1}^b \nu_{\i} = n$.  Here $\nu_{\i}$ will be the number of
vertices in the $\i$th connected component of the interesection.
(We should divide by
$b!$ to account for our assumption that the connected components
are labelled).
\item \label{step:EY2-2}
Choose a partition of the $n$ vertices into $b$ groups, where the
size of the $\i$th group is $\nu_{\i}$.
\item \label{step:EY2-3} In each group, choose a spanning tree on that
group and choose a partial pairing that projects to that tree.
This specifies a component of the intersection.
\end{enumerate}

Now, collapse the buckets in each group into a single super-bucket,
giving exactly $b$ super-buckets.  The
$\i$th super-bucket has $\k\nu_{\i}-2\left(\nu_{\i}-1\right)$
unpaired prevertices. We now want to pair up two pair-disjoint spanning
trees $T_{1}',T_{2}'$ in the collapsed pairing model. These will
extend to $T_{1}$ and $T_{2}$ using the intersection subtrees chosen
in Step \ref{step:EY2-3}.\vspace{-5pt}

\begin{enumerate}
\setcounter{enumi}{4}
\item 
\label{step:EY2-4}
For $\i=1,\ldots, b$, choose $\r_{\i}$ and
$\w_{\i}$, the degree of vertex $\i$ in $T_{1}'$ and $T_{2}'$
respectively, in such a way that 
$\r_{\i}+\w_{\i}\le\k\nu_{\i}-2\left(\nu_{\i}-1\right)$.
These must also satisfy
\[ \sum_{\i=1}^b \r_{\i} = \sum_{\i=1}^b \w_{\i} = 2(b-1),\]
as they are the degree sequence of a spanning tree on $b$ vertices.
\item \label{step:EY2-5}
Choose two trees $T_{1}',T_{2}'$ on the $b$ vertices
that are consistent with the degree sequences chosen in Step \ref{step:EY2-5}.
\item \label{step:EY2-6}
Pair up these two trees in a pair-disjoint way.
\item \label{step:EY2-7}
Pair all remaining prevertices to complete a $\k$-regular pairing.
\end{enumerate}

Given $\nu$, the number of ways to complete Step \ref{step:EY2-2}
is 
\[
\s_{\ref{step:EY2-2}}= \binom{n}{\nu_{1},\ldots,\nu_{b}}.
\]
By (\ref{useful}), 
the number of ways to complete Step \ref{step:EY2-3} is 
\[
\s_{\ref{step:EY2-3}}=
\k^{n}\, \prod_{\i=1}^{b}\, \left(\nu_{\i}-2\right)!\,
\binom{(\k-1)\nu_{\i}}{\nu_{\i}-2}.
\]
The number of ways to complete Step \ref{step:EY2-5} is 
\[
\s_{\ref{step:EY2-5}}= \binom{b-2}{\r_{1}-1,\ldots,\r_{b}-1}\,
  \binom{b-2}{\w_{1}-1,\ldots,\w_{b}-1},
\]
 by (\ref{eq:moon}), and the number of ways to complete Step \ref{step:EY2-6}
is 
\[
\s_{\ref{step:EY2-6}}=\prod_{\i=1}^{b}\falling{(\k-2)\nu_{\i}+2}{\r_{\i}+\w_{\i}}.
\]
Finally, for Step \ref{step:EY2-7} there are 
\[
\sum_{\i=1}^{b}\left(\k\nu_{\i}-2\left(\nu_{\i}-1\right)-\r_{\i}-\w_{\i}\right) 
  =  \left(\k-2\right)n-2\left(b-2\right)
\]
prevertices remaining, so the number of ways to complete Step \ref{step:EY2-7}
is 
\[
\s_{\ref{step:EY2-7}}=\nP\left(\left(\k-2\right)n-2\left(b-2\right)\right).
\]
For this construction to make sense, the quantity $(\k-2)n-2(b-2)$ must be 
nonnegative. This is certainly true when $\k\geq 4$, but when $\k=3$ this
imposes the constraint that $b\leq n/2+2$.  (This explains the definition
of $N(n,\k)$ in the lemma statement.)

It will be convenient to work with the nonnegative variables
\begin{equation}
\label{translate} \rr_{\i}=\r_{\i}-1, \quad \ww_{\i}=\w_{\i}-1 \quad \text{ and }
 \quad \q_{\i}=\left(\k-2\right)\nu_{\i}-\rr_{\i}-\ww_{\i}
\end{equation}
defined for $\i=1,\ldots, b$.
Let 
\[ \S_{\ref{step:EY2-1}}(b) = \{ \nu \in \{ 1,\ldots, n\}^b : 
       \sum_{\i=1}^b \nu_{\i} = n \}
\]
be the set of possible sequences $\nu$
from Step \ref{step:EY2-1} and let 
\begin{align*} \S_{\ref{step:EY2-4}}(\nu) 
 = \left\{ \left(\rr,\ww,\q\right)\in\left(\mathbb{N}^{b}\right)^{3}: \right. &
   \left.
  \quad\rr_{\i}+\ww_{\i}+\q_{\i}=\left(\k-2\right)\nu_{\i} \, \mbox{ for }\,
  \i=1,\ldots, b,\right.\\
 &  \left. {} \quad 
   \sum_{j=1}^b \rr_{\i} = \sum_{j=1}^b \ww_{\i} = b-2\right\} 
\end{align*}
be the set of sequences arising from Step~\ref{step:EY2-4} using (\ref{translate}).

Combining all of the above gives
\begin{align*}
\left|\P\right|\,\E\Y^{2} & = \sum_{b=1}^{N(n,\k)}\, 
 \frac{1}{b!}\, \sum_{\nu\in\S_{\ref{step:EY2-1}}(b)}\, 
  \s_{\ref{step:EY2-2}}\s_{\ref{step:EY2-3}}\,
  \sum_{\left(\rr,\ww,\q\right)\in\S_{\ref{step:EY2-4}}(\nu)}\,
  \s_{\ref{step:EY2-5}}\s_{\ref{step:EY2-6}}\s_{\ref{step:EY2-7}}\\
 & = \sum_{b=1}^{N(n,\k)}\, \frac{n!\, \k^{n}}
  {b!\,\left(\left(\k/2-1\right)n-b+2\right)!\,2^{\left(\k/2-1\right)n-b+2}}\, 
\sum_{\nu\in\S_{\ref{step:EY2-1}}(b)}\,
 \prod_{\i=1}^{b}\, \frac{\left(\left(\k-1\right)\nu_{\i}\right)!}{\nu_{\i}!}\\
 &  \quad\times\sum_{\left(\rr,\ww,\q\right)\in\S_{\ref{step:EY2-4}}(\nu)}
 \, \binom{b-2}{\rr_{1},\ldots,\rr_{b}}\, \binom{b-2}{\ww_{1},\ldots,\ww_{b}}
  \,\binom{(\k-2)n-2(b-2)}{\q_{1},\dots,\q_{b}}.
\end{align*}
Now
\begin{align*}
 &  \sum_{\left(\rr,\ww,\q\right)\in\S_{\ref{step:EY2-4}}(\nu)}\,
  \binom{b-2}{\rr_{1},\ldots,\rr_{b}}\,
  \binom{b-2}{\ww_{1},\ldots,\ww_{b}}\,
  \binom{(\k-2)n-2(b-2)}{\q_{1},\ldots,\q_{b}}\\
 & = \sum_{\left(\rr,\ww,\q\right)\in\S_{\ref{step:EY2-4}}(\nu)}\,
  \left[z_{1}^{\rr_{1}}\dots z_{b}^{\rr_{b}}\right]\,
   \left(\sum_{\i=1}^{b}z_{\i}\right)^{b-2}\\
 &  \quad\times\left[z_{1}^{\ww_{1}}\dots z_{b}^{\ww_{b}}\right]\left(\sum_{\i=1}^{b}z_{\i}\right)^{b-2}\left[z_{1}^{\q_{1}}\dots z_{b}^{\q_{b}}\right]\left(\sum_{\i=1}^{b}z_{\i}\right)^{\left(\k-2\right)n-2\left(b-2\right)}\\
 & = \left[ z_1^{(\k-2)\nu_1}\, z_2^{(\k-2)\nu_2} \cdots z_b^{(\k-2)\nu_b}
  \right]\, \left(\sum_{\i=1}^{b}z_{\i}\right)^{\left(\k-2\right)n}\\
 & = \binom{(\k-2)n}{(\k-2)\nu_{1},\ldots,(\k-2)\nu_{b}}.
\end{align*}
It follows that
\begin{align*}
& \left|\P\right|\,\E\Y^{2} \\
  &= \frac{n!\, \left((\k-2)n\right)!\, \k^n}{2^{(\k/2-1)n+2}}\,
  \sum_{b=1}^{N(n,\k)} \, \frac{2^b}{b!\,\left(\left(\k/2-1\right)n-b+2\right)!}
 \, \sum_{\nu\in\S_{\ref{step:EY2-1}}(b)}\,
   \prod_{\i=1}^{b}\, \binom{(\k-1)\nu_{\i}}{\nu_{\i}}. 
\end{align*}
This is equal to the expression in the lemma statement,
by definition of $\S_{\ref{step:EY2-1}}(b)$. 
\end{proof}

We now seek to evaluate
\[ [z^n] \left( \sum_{\i=1}^\infty \binom{(\k-1)\i}{\i} z^{\i}\right)^b.\]
By Stirling's approximation and the ratio test, the radius of
convergence of the series $\sum_{\i=1}^{\infty}\, \binom{(\k-1)\i}{\i}\, z^{\i}$
equals $\frac{\left(\k-2\right)^{\k-2}}{\left(\k-1\right)^{\k-1}}$.
Hence, 
\[
\f\left(z\right):=\sum_{\i=1}^{\infty}\,
 \binom{(\k-1)\i}{\i}\, \left(\frac{\left(\k-2\right)^{\k-2}}{\left(\k-1\right)^{\k-1}}\right)^{\i}z^{\i}
\]
 is analytic in the disk $\left\{ z:\left|z\right|<1\right\} $. Define
$\beta=b/n$ and let $r_{\beta}\in\left(0,1\right)$ be fixed
for each $\beta$ (we will determine this later). Then, with the contour
$\Gamma:\left[-\pi,\pi\right]\to\mathbb{C}$ defined by 
$\theta\mapsto r_{\beta}e^{i\theta}$, we have
\begin{align}
\left[z^{n}\right]\f\left(z\right)^{b} & = 
  \frac{1}{2\pi i}\, \int_{\Gamma}\frac{\f\left(z\right)^{b}}{z^{n+1}}\,
    \mathrm{d}z\nonumber \\
 & = \frac{1}{2\pi}\int_{-\pi}^{\pi}\left(\frac{\f\left(r_{\beta}e^{i\theta}\right)^{\beta}}{r_{\beta}e^{i\theta}}\right)^{n}\mathrm{d}\theta,\label{eq:contour-integral}
\end{align}
by Cauchy's coefficient formula.
Let
\[ \X_{n}=\left\{ \dfrac{1}{n},\dfrac{2}{n},\dots,\dfrac{N(n,\k)}{n}\right\} 
   \times\left[-\pi,\pi\right]
\]
be the sample space for pairs $\left(\beta,\theta\right)$, 
and define
$\g:\X_{n}\to\mathbb{C}$ by
\begin{equation}
\label{andef}
\g\left(\beta,\theta\right)=\frac{2^{\beta n}}{\left(\beta n\right)!\,\left(\left(\k/2-1\right)n-\beta n+2\right)!}\left(\frac{\f\left(r_{\beta}e^{i\theta}\right)^{\beta}}{r_{\beta}e^{i\theta}}\right)^{n}.
\end{equation}
Finally, let
\begin{equation}
\label{eq:E}
\H=\sum_{b=1}^{N(n,\k)}\int_{-\pi}^{\pi}\g\left(b/n,\theta\right)\,\mathrm{d}\theta.
\end{equation}
Then, by Lemma~\ref{lem:|P|EY2} and (\ref{eq:contour-integral}), 
\begin{equation}
\left|\P\right|\,\E\Y^{2} = 
\frac
{\left(\k-1\right)^{n\left(\k-1\right)}n!\,\left(\left(\k-2\right)n\right)!\,\k^{n}}
{2\pi\left(\k-2\right)^{n\left(\k-2\right)}2^{\left(\k/2-1\right)n+2}}\, \H .
\label{eq:E2}
\end{equation}

We now apply the saddle point method to estimate the sum in
(\ref{eq:E}) in the case that $\k=3$. 
Our proof is adapted from that of~\cite[Theorem 2.3]{GJR}.

When $\k=3$ the function $f$ satisfies 
\[
\f\left(z\right)=\sum_{\i=1}^{\infty}\, \binom{2\i}{\i}\, 
  \left(\frac{z}{4}\right)^{\i}=\left(1-z\right)^{-1/2}-1.
\]
We note for later that if $\theta\in [-\pi,\pi]$ is nonzero then
\begin{equation}
 \left| f\left( r_\beta e^{i\theta}\right)\right| = 
   \left| \, \sum_{j=1}^\infty 
     \binom{2j}{j}\, \left(\frac{r_\beta}4\right)^j
 e^{ij\theta}\, \right|
   < 
\sum_{j=1}^\infty 
     \binom{2j}{j}\, \left(\frac{r_\beta}4\right)^j
 = |f(r_\beta)|,
\label{lem:technical}
\end{equation} 
using the triangle inequality.
Hence for each $\beta$ the function $\theta \mapsto |f( r_\beta e^{i\theta})|$
on $[-\pi,\pi]$ is uniquely maximised at $\theta =0$.

Define $\X=\left(0,1\right]\times\left[-\pi,\pi\right]$ 
and let $\X^{*}\subset\X$ be a set (to be determined) such that for $\left(\beta,\theta\right)\in\X^{*}$,
both $\beta$ and $\k/2-1-\beta=1/2-\beta$ are bounded below by some
positive constant. 
Then Stirling's approximation gives, for $\beta\in\X^*\cap\X_n$,
\begin{align}
& \g\left(\beta,\theta\right) \nonumber\\
 & \sim \frac{e^{n/2}}{2\pi\, n^{n/2+3}\, \sqrt{\beta\left(1/2-\beta\right)}\, \left(1/2-\beta\right)^{2}}\left(\frac{\left(2\f\left(r_{\beta}e^{i\theta}\right)\right)^{\beta}}{r_{\beta}e^{i\theta}\beta^{\beta}\left(1/2-\beta\right)^{\left(1/2-\beta\right)}}\right)^{n}.\label{eq:g-asymptotic-working}
\end{align}
Next, define the half-spaces $\X^{1/2}=\left(0,1/2\right]\times\left[-\pi,\pi\right]$
and $\bar{\X}^{1/2}=\left[0,1/2\right]\times\left[-\pi,\pi\right]$.
Define the real-valued sequence $\left(c_{n}\right)_{n\in\mathbb{N}}$
and the functions $\psi:\bar{\X}^{1/2}\to\mathbb{R}$ and $\phi:\X^{1/2}\to\mathbb{C}$
by 
\begin{align*}
c_{n} & = \frac{e^{n/2}}{2\pi n^{n/2+3}},\\
\psi\left(\beta,\theta\right) & = \beta^{-1/2}\left(1/2-\beta\right)^{-5/2},\\
\phi\left(\beta,\theta\right) & = \beta\log\left(2\f\left(r_{\beta}e^{i\theta}\right)\right)-\log r_{\beta}-i\theta-\beta\log\beta-\left(1/2-\beta\right)\log\left(1/2-\beta\right),
\end{align*}
so that we have 
\begin{equation}
\g\left(\beta,\theta\right)\sim c_{n}\, \psi\left(\beta,\theta\right)\,
 e^{n\phi\left(\beta,\theta\right)}\label{eq:g-asymptotic-statement}
\end{equation}
uniformly for $\left(\beta,\theta\right)\in\X_{n}\cap\X^{*}$.

Let $D$ denote the differential operator 
\[ \left(D\phi\left(x\right)\right)_{\i}=
   \frac{\partial\phi\left(x\right)}{\partial x_{\i}}.\]
We seek a stationary point of $\phi$.  The condition
$\frac{\partial\phi\left(\beta,0\right)}{\partial\theta}=0$
is equivalent to the condition $\beta r_\beta f'(r_\beta) = f(r_\beta)$.
Solving for $r_\beta$ gives
\[
r_{\beta}=\frac{1}{8}\left(8-4\beta-\beta^{2}\pm\sqrt{\beta^{3}\left(8+\beta\right)}\right).
\]
We choose
$r_{\beta}=\frac{1}{8}\left(8-4\beta-\beta^{2}-
   \sqrt{\beta^{3}\left(8+\beta\right)}\right)\in\left(0,1\right)$,
which ensures that
$\frac{\partial\phi\left(\beta,0\right)}{\partial\theta}=0$.
Next, we calculate that with this choice of $r_\beta$,
\[ \frac{\partial\phi}{\partial \beta}(\beta,0) = 
  \log\left(\frac{\left(4-\beta-\sqrt{\beta(8+\beta)}\right)(1-2\beta)}
                 {\beta\left(\beta + \sqrt{\beta(8+\beta)}\right)}\right).
\]
Setting this equal to 0 and solving for $\beta$ gives the equation
$(3\beta-1)(\beta^2-4\beta+2)=0$.
The only solution with $\beta\in\left(0,\frac12\right]$ is $\beta=\frac13$ so we choose
$\x^{*}=\left(\frac{1}{3},0\right)$ and check that 
$D\phi\left(\x^{*}\right)=0$.

Note that $\phi\left(\x^{*}\right)=\log\left(4\sqrt{\frac{2}{3}}\right)$,
and
\[
H=-\begin{pmatrix}\frac{63}{5} & 0\\
0 & \frac{5}{2}
\end{pmatrix}
\]
is the Hessian matrix of $\phi$ at $\x^{*}$. Define $\C_{1}=5/8$,
so that $-4\C_{1}$ is the largest eigenvalue of $H$.

Now, define $\hat{\phi}$ by $\hat{\phi}\left(\x\right)=\phi\left(\x\right)-\phi\left(\x^{*}\right)$,
and define $\gg:\X_{n}\to\mathbb{C}$ by 
\[
\gg\left(\x\right)=
 c_n^{-1}\, e^{-n\phi\left(\x^{*}\right)}\, 
  \g\left(\x\right).
\]
With a Taylor expansion about $\x^{*}$, for $\x\in\X^{1/2}$ we have
\begin{equation}
\hat{\phi}\left(\x\right) = \frac{1}{2}\left(\x-\x^{*}\right)^{T}H\left(\x-\x^{*}\right)+h\left(\x\right)\left|\x-\x^{*}\right|^{2},\label{eq:taylor}
\end{equation}
where $h\left(\x\right)$ is complex and $h\left(\x\right)\to0$ as
$\x\to\x^{*}$. For all $v\in\mathbb{R}^{2}$ we have $v^{T}Hv\le-2\C_{1}\left|v\right|^{2}$,
so we can choose $\xi<\frac{1}{6}$ such that $\Re\hat{\phi}\left(\x\right)\le-\C_{1}\left|\x-\x^{*}\right|^{2}$
for $\left|\x-\x^{*}\right|<\xi$. Define $\X^{*}=\left\{ \x\in\X^{1/2}:\left|\x-\x^{*}\right|<\xi\right\} $,
satisfying the requirement for (\ref{eq:g-asymptotic-working}).

Next, define the sets 
\begin{align*}
\Xk 1 & = \left\{ \x\in\X^{*}:\left|\x-\x^{*}\right|<n^{-1/3}\right\} ,\\
\Xk 2 & = \X^{*}\setminus \Xk 1,\\
\Xk 3 & = \X^{1/2}\setminus \X^{*},\\
\Xk 4 & = \X\setminus \X^{1/2},
\end{align*}
so that with
\[
\Hk{\i}=\sum_{b=1}^{n/2+2}\int_{-\pi}^{\pi}\gg\left(b/n,\theta\right)\,
  \one_{\Xk{\i}}(b/n,\theta)\, \mathrm{d}\theta
\]
we have 
\begin{equation}
\label{split}
 c_n^{-1}\, e^{-n\phi\left(\x^{*}\right)} \H=\Hk 1+\Hk 2+\Hk 3+\Hk 4.
\end{equation}

\begin{lemma}
With notation as above, we have
\[ \Hk 1 + \Hk 2 + \Hk 3 + \Hk 4 \sim \Hk 1
   \sim \frac{144\pi}{\sqrt{7}}.\]
\label{lem:dominates}
\end{lemma}

\begin{proof}
Note that $\psi$ is a continuous function defined on a compact set.
So, $\psi$ is absolutely bounded on its domain, by $\C_{2}$ say.
By (\ref{eq:g-asymptotic-statement}), it follows that $\left|\gg\left(\x\right)\right|=O\left(e^{n\Re\hat{\phi}\left(\x\right)}\right)$
uniformly for $\x\in\X^{*}$. For $\x\in\Xk 2$ we have 
\[
n\Re\hat{\phi}\left(\x\right)\le-n\C_{1}\left|\x-\x^{*}\right|^{2}\le-\C_{1}n^{1/3}\to-\infty
\]
and consequently
\begin{equation}
\left|\Hk 2\right|  \, \leq \, 
  \sum_{b=1}^{n} \int_{-\pi}^{\pi}\, \left|\gg\left(b/n,\theta\right)\right|\,
  \one_{\Xk 2}(b/n,\theta)\, \mathrm{d}\theta \,
  = \, O\left(ne^{-\C_{1}n^{1/3}}\right) = o(1).
\label{F2bound}
\end{equation}

Now (\ref{lem:technical}) implies that
for each $\beta$, $\Re\hat{\phi}\left(\beta,\theta\right)$
is uniquely maximized when $\theta=0$. 
Also
$\frac{\partial\Re\hat{\phi}}{\partial\beta}\left(\beta,0\right)=\frac{\partial\hat{\phi}}{\partial\beta}\left(\beta,0\right)=0$ only for $\left(\beta,0\right)=\x^*$, since $\hat{\phi}$ is real along the line $\theta=0$.
Checking the values of $\Re\hat{\phi}(\beta,0)$ in the limit as $\beta\to 0$ 
and $\beta\to \nfrac{1}{2}$, 
it follows that $\Re\hat{\phi}$ attains a unique
maximum on $\X^{1/2}$ at $\x^{*}$. 
Let $-\C_{3}<0$ be the maximum value of $\Re\hat{\phi}$ on $\Xk 3$.

Let $u\lor w=\max\left\{ u,w\right\} $ for real numbers $u,w$. 
We now redo the calculations of (\ref{eq:g-asymptotic-working}) using an
alternate form of Stirling's inequality which holds for all $k\geq 0$,
namely $\sqrt{k\lor1}\left(\frac{k}{e}\right)^{k}\le k!$.
For $\left(\beta,\theta\right) \in \X^{1/2}\cap \X_n$,
\begin{align*}
\left|\gg\left(\beta,\theta\right)\right| & \le 
\frac{ne^{2}}{\sqrt{\left(\beta n\lor 1\right)
  \left(\left(n/2-\beta n+2\right)\lor 1\right)}
  \left(1/2-\beta\right)^{2}}\, e^{n\Re\hat{\phi}\left(\x\right)}\\
 & = e^{n\Re\hat{\phi}\left(\x\right)+o\left(n\right)}.
\end{align*}
It follows that
\begin{equation}
\label{F3bound}
\left|\Hk 3\right|=O\left(ne^{-\C_{3}n/2}\right) = o(1).
\end{equation}

Next suppose that $\left(\beta,\theta\right)\in\Hk 4\cap \X_n$.
Then we have $\frac{1}{2}<\beta\le\frac{1}{2}+o(1)$ and $\left(n/2-b+2\right)!=1$. 
By the alternate form of Stirling's inequality and (\ref{lem:technical}),
\begin{align*}
\left|\g\left(\beta,\theta\right)\right| & \le \frac{n^{3}e^{2}}{\sqrt{\left(\beta n\lor 1\right)}}\, c_{n}\left|\frac{\left(2\f\left(r_{\beta}e^{i\theta}\right)\right)^{\beta}}{r_{\beta}e^{i\theta}\beta^{\beta}}\right|^{n}\\
& \le e^{o\left(n\right)}c_{n}\left(\frac{\left(2\f\left(r_{1/2}\right)\right)^{1/2}}{r_{1/2}\left(\frac12\right)^{1/2}}\right)^{n}.
\end{align*}
By direct computation,
\[
\log\frac{\left(2\f\left(r_{1/2}\right)\right)^{1/2}}{r_{1/2}\left(\frac12\right)^{1/2}}=\phi(x^*)+\C_4
\]
for some $\C_4>0$. It follows that
\begin{equation}
\label{F4bound}
\left|\Hk 4\right|=O\left(ne^{-\C_{4}n/2}\right) = o(1).
\end{equation}
It remains to consider $\Hk 1$. Define $\ceil{\left(\beta,\theta\right)}=\left(\frac{\ceil{\beta n}}{n},\theta\right)$,
so that $\H=n\int_{\X}\gg\left(\ceil{\x}\right)\mathrm{d}\x$. For
any $y\in\mathbb{R}^{2}$, define $\x_{\y}=\x^{*}+\y/\sqrt{n}$ and
$B_{n}=\left\{ \y:\ceil{\x_{\y}}\in\Xk 1\right\} $, so that we can
make the change of variables
\[
\Hk 1=\int_{B_{n}}\gg\left(\ceil{\x_{\y}}\right)\mathrm{d}\y.
\]
Note that 
\[
\left|\y/\sqrt{n}\right|=\left|\x_{\y}-\x^{*}\right|\le\left|\x_{\y}-\ceil{\x_{\y}}\right|+\left|\ceil{\x_{\y}}-\x^{*}\right|=O\left(n^{-1/3}+n^{-1}\right)=O\left(n^{-1/3}\right)
\]
for $\y\in B_{n}$, so that $B_{n}$ is approximately a ball of radius
$O\left(n^{1/6}\right)$.

Next, a first-order Taylor expansion of $D\phi$ about $\x^{*}$ gives
\[
\left|D\phi\left(\x_{\y}\right)\right|=O\left(\left|\y/\sqrt{n}\right|\right)=O\left(n^{-1/3}\right).
\]
Another first-order Taylor expansion of $\phi$ about $\x_{\y}$ gives
\[
\phi\left(\ceil{\x_{\y}}\right)-\phi\left(\x_{\y}\right)=O\left(\left|D\phi\left(\x_{\y}\right)\right|\left|\ceil{\x_{\y}}-\x_{\y}\right|\right)=O\left(n^{-4/3}\right),
\]
so that 
\begin{equation}
e^{n\phi\left(\x_{y}\right)}\sim e^{n\phi\left(\ceil{\x_{y}}\right)}\label{eq:ceil-asymptotic}
\end{equation}
uniformly. Now, for each $\y\in\mathbb{R}^{2}$, we have $\ceil{\x_{\y}}\to\x^{*}$.
For $n$ large enough so that $\y\in B_{n}$, we have $\psi\left(\ceil{\x_{\y}}\right)\to\psi\left(\x^{*}\right)$
by continuity and $e^{n\phi\left(\x_{\y}\right)}\to e^{\frac{1}{2}\y^{T}H\y}$
by (\ref{eq:taylor}). We therefore have $\one_{B_{n}}\left(\y\right)\gg\left(\ceil{\x_{\y}}\right)\to\psi\left(\x^{*}\right)e^{\frac{1}{2}\y^{T}H\y}$
for all $\y$.

Recalling that $\C_{2}$ and $2\C_{1}$ are bounds involving $\psi$
and $\phi$ respectively, with (\ref{eq:ceil-asymptotic}) we have
$\left|\one_{B_{n}}\left(\y\right)\gg\left(\ceil{\x_{\y}}\right)\right|\le2\C_{2}e^{-\C_{1}\left|\y\right|^{2}}$
for sufficiently large $n$. 
Since
\[
\left(\det\left(-H\right)\right)^{-1/2}  = \frac{1}{3}\sqrt{\frac{2}{7}},\quad
\psi\left(\x^{*}\right)  =  108\sqrt{2}
\]
we obtain, by the dominated convergence theorem,
\[
\Hk 1\to\psi\left(\x^{*}\right)\int_{\mathbb{R}^{2}}\, e^{\frac{1}{2}\y^{T}H\y}\, 
  \mathrm{d}\y=2\pi\, \psi\left(\x^{*}\right)\, 
     \left(\det\left(-H\right)\right)^{-1/2} = \frac{144\pi}{\sqrt{7}}.
\]
Combining this with (\ref{F2bound}--\ref{F4bound})
completes the proof.
\end{proof}

We now pull these calculations together to prove the following.

\begin{lemma}
Let $\k=3$. Then
\[ \E[Y^2] \sim \frac{18}{\sqrt{14}}\, \left(\frac{16}{3}\right)^n,\]
and hence
\[ \frac{\E[Y^2]}{[\E Y]^2} \rightarrow \frac{9}{\sqrt{14}}.
\]
It follows that Condition~\ref{cond:A4} holds when $\k=3$.
\label{lem:secondmoment}
\end{lemma}

\begin{proof}
Lemma~\ref{lem:dominates} and (\ref{split}) prove that
\[ \H \sim \frac{72}{n^3\sqrt{7}}\, \left( 4\sqrt{\frac{2e}{3n}}\right)^n.\]
Substituting $\k=3$ into (\ref{eq:E2}) and applying
(\ref{eq:num-pairings-asymptotic}) gives 
\[
\E\Y^{2} = \frac{(6\sqrt{2})^n\, (n!)^2}{4\pi\, \nP(3n)}\, \H
      \sim\frac{18}{\sqrt{14}}\left(\frac{16}{3}\right)^{n},
\]
using Stirling's approximation. 
Then, with (\ref{eq:EY-asymptotic}) and Lemma~\ref{EY2/EY2-target},
we conclude that
\[
\frac{\E\Y^{2}}{(\E\Y)^{2}}\to\frac{9}{\sqrt{14}}=\exp\left(\sum_{\i=1}^{\infty}\lambda_{\i}\gamma_{\i}^{2}\right).
\]
This establishes Condition~\ref{cond:A4}, as required. 
\end{proof}

We can now complete the proof of Theorem~\ref{thm:distribution-3}.

\begin{proof}[Proof of Theorem \ref{thm:distribution-3}]
We will prove that for general $\k \geq 3$, if Condition~\ref{cond:A4}
holds then Conjecture~\ref{cnj:distribution} is true.
In particular, this will prove Theorem~\ref{thm:distribution-3},
using Lemma~\ref{lem:secondmoment}.

Suppose that Condition \ref{cond:A4} is satisfied for some fixed
integer $\k\geq 3$. Then by Lemma~\ref{EY2/EY2-target} we may apply
Theorem~\ref{thm:janson} to conclude that
(\ref{eq:W}) holds for $Y$.  Therefore, for all real numbers $y$ we have
\begin{alignat*}{2}
\Pr\left(Y_{\mathcal{G}}/\E Y_{\mathcal{G}}<y\right) 
  &=\enspace&& \Pr\left(Y/\E Y_{\mathcal{G}}<y|X_{1}=X_{2}=0\right)\\
 & \rightarrow &&
  \Pr\left(W\exp\left(\lambda_{1}\wea_{1}+\lambda_{2}\wea_{2}\right)<y|
           Z_{1}=Z_{2}=0\right)\\
 &=&& \Pr\left(\prod_{\i=3}^{\infty}\left(1+\wea_{\i}\right)^{Z_{\i}}
        e^{-\lambda_{\i}\wea_{\i}}<y\right).
\end{alignat*}
Hence Conjecture \ref{cnj:distribution} is a consequence
of \ref{cond:A4}. 
\end{proof}

\subsection{Support for Conjecture~\ref{cnj:distribution}}\label{ss:numerical}
\global\long\def\a{p}

Let $\a_{\k}\left(n\right)$ denote the quotient 
\[
\frac{\E Y^{2}}{\left(\E Y\right)^{2}}\left/\exp\left(\sum_{\i=1}^{\infty}\lambda_{\i}\wea_{\i}^{2}\right)\right. .
\]
For any fixed integer $\k\geq 4$, Conjecture~\ref{cnj:distribution} holds 
if and only Condition~\ref{cond:A4} from Theorem~\ref{thm:janson} is satisfied;
that is, if and only if $\a_{\k}(n)\sim 1$.
Using Lemma~\ref{lem:|P|EY2} we can compute
$\a_{\k}\left(100\right)$ for various values of $\k$:
\begin{center}
\begin{tabular}{|c|c|c|c|c|c|}
\hline 
$\k$ & 3 & 4 & 5 & 6 & 100\tabularnewline
\hline 
$\a_{\k}\left(100\right)$ & 0.9761 & 0.9881 & 0.9921 & 0.9942 & 0.9998\tabularnewline
\hline 
\end{tabular}
\par\end{center}
Figure \ref{fig:EY2-plot} is a plot of
$\a_{\k}\left(n\right)$ for $\k\in \{ 3,4,5,6,100\}$ and $n\leq 50$.  

\begin{figure}[ht!]
\begin{center}
\centerline{\includegraphics{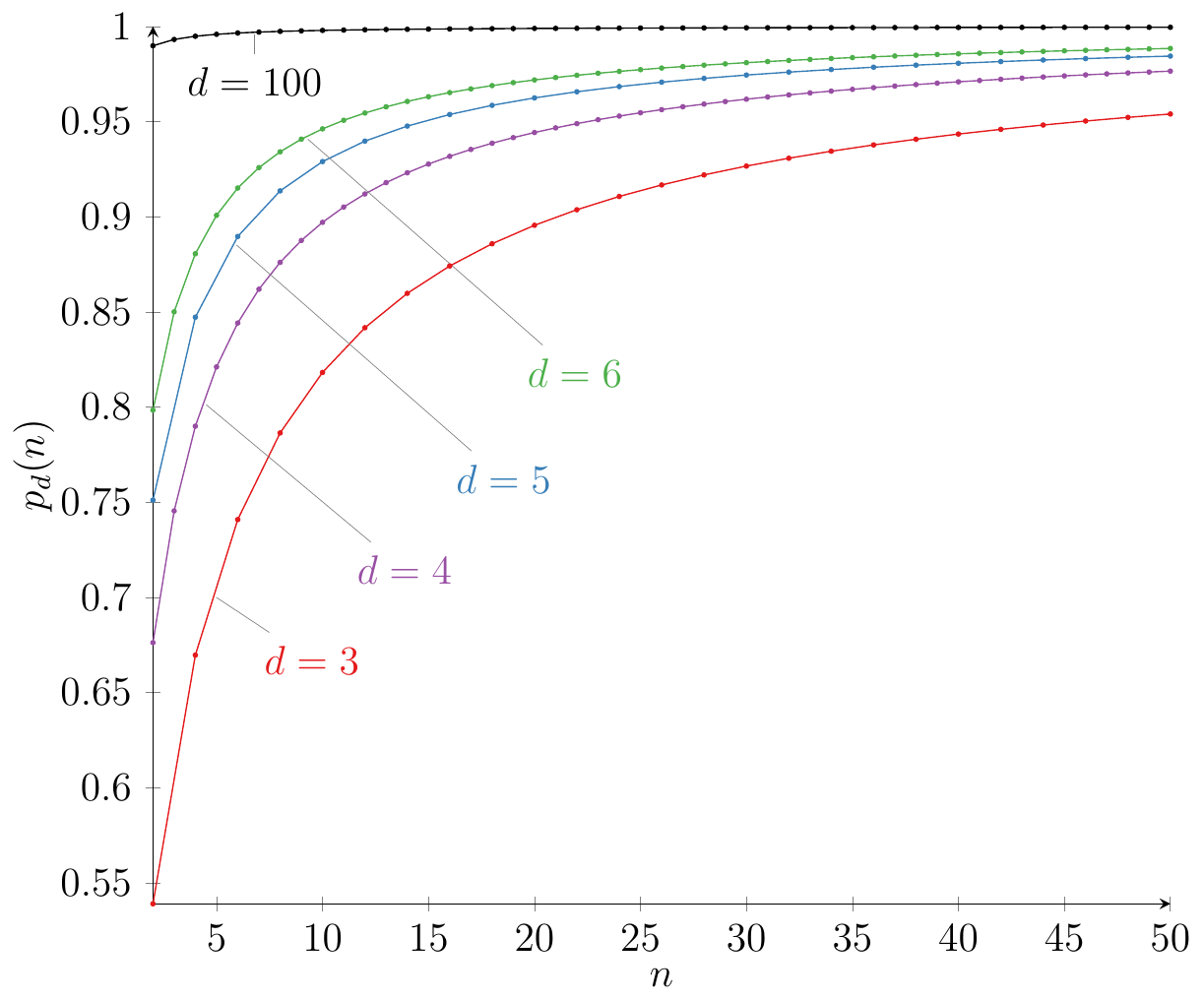}}
\caption{\label{fig:EY2-plot} A plot of $\a_\k(n)$ for $\k \in \{ 3,4,5,6,100\}$}
\end{center}
\end{figure}

This plot supports our conjecture that
$\a_{\k}(n)\sim 1$ for all $\k\geq 4$. 
Indeed, the rate of convergence to 1 appears to increase as $\k$ increases.

We now give an asymptotic result which is equivalent to 
Conjecture~\ref{cnj:distribution}.
Combining (\ref{eq:num-pairings-asymptotic}), \ref{cond:A4}, (\ref{eq:EY-asymptotic}), Lemma~\ref{EY2/EY2-target}, Lemma~\ref{lem:|P|EY2} and applying Stirling's 
formula shows that for a fixed integer $\k\geq 4$,
Conjecture~\ref{cnj:distribution} holds if and only if
\begin{align}
\sum_{b=1}^{n}\,&
  \frac{2^{b}}{b!\,\left(\left(\k/2-1\right)n-b+2\right)!}\,
  \left[z^{n}\right]\,\left(\sum_{\i=1}^{\infty}\binom{(\k-1)\i}{\i}\, z^{\i}\right)^{b} \notag\\
&\sim
  \frac{2\k^{2}}{\pi\left(\k-2\right)^{4} \, n^3}\,\,
  \sqrt{\frac{2\k-2}{\k^{2}-\k+1}}\,\,
   \left(\frac{\left(\k-1\right)^{2(\k-1)}}{\left(\k-2\right)^{2(\k-2)}}\left(\frac{2e}{\k n}\right)^{\k/2-1}\right)^{n}.
\label{sufficient}
\end{align}


\section*{Acknowledgements} 
We would like to thank the referee for their helpful comments.

\end{document}